\documentclass[preprint]{elsarticle}
\usepackage[english]{babel}
\usepackage{amsmath}
\usepackage{amssymb}
\usepackage{pst-3d,float}

\newtheorem{theorem}{Theorem}[section]

\newtheorem{corollary}{Corollary}[section]

\newtheorem{definition}{Definition}[section]
\newtheorem{example}[theorem]{Example}

\newtheorem{lemma}{Lemma}[section]

\newtheorem{proposition}[theorem]{Proposition}
\newtheorem{remark}{Remark}[section]

\newenvironment{proof}[1][Proof]{\textbf{#1.} }{\ \rule{0.5em}{0.5em}}

\newcommand{\nc}{\newcommand}
\nc{\Nn}{\hbox{I\!N}}
\nc{\R}{\mathbb{R}}\nc{\rmd}{\mathrm{d}}\nc{\rme}{\mathrm{e}}\nc{\rmi}{\mathrm{i}}

\nc{\dps}{\displaystyle} \nc{\un}{\underline} \nc{\Bb}{$\Box$}
\nc{\vareps}{\un{\un{\varepsilon}}} \nc{\eps}{\varepsilon}
\nc{\tu}{\widetilde{u}} \nc{\te}{\widetilde{e}}
\nc{\tpsi}{\widetilde{\Psi}} \nc{\ov}{\overline{v}}\nc{\N}{\mathbb{N}}
\nc{\oV}{\overline{V}} \nc{\wL}{\widetilde{L}} \nc{\bx}{{\bf x}}
\makeatletter

\@addtoreset{equation}{section} \makeatother
\begin{document}


\begin{frontmatter}
\title{Explicit construction of operator scaling Gaussian random fields \tnoteref{label1}}


\author{M.Clausel\fnref{label2}}

\address[label2]{Laboratoire d'Analyse et   de
Math\'ematiques Appliqu\'ees,   UMR 8050 du CNRS, Universit\'e Paris Est,
             61 Avenue du G\'en\'eral de Gaulle,
             94010 Cr\'eteil Cedex,
            France.\\e-mail : clausel@univ-paris12.fr\\Tel : 01 45 17 17 61}
\author{B.Vedel\fnref{label3}}
\address[label3]{Laboratoire de Mathematiques et Applications des Math\'ematiques, Universite de Bretagne Sud, Universit\'e Europ\'eene de Bretagne
Centre Yves Coppens, Bat. B, 1er et., Campus de Tohannic BP 573,
56017 Vannes, France.\\e-mail: vedel@univ-ubs.fr\\Tel : 02 97 01 71 51}

\newpage

\begin{abstract}
We propose an explicit way to generate a large class of Operator
scaling Gaussian random fields (OSGRF). Such fields are
anisotropic generalizations of self-similar fields. More
specifically, we are able to construct any Gaussian field
belonging to this class with given Hurst index and exponent. Our
construction provides - for simulations of texture as well as for
detection of anisotropies in an image - a large class of models
with controlled anisotropic geometries and structures.
\end{abstract}


\end{frontmatter}

\newpage

\section{Introduction}\label{SecOne}
Random fields are a useful tool for modelling spatial phenomena
such as environmental fields, including for example, hydrology,
geology, oceanography and medical images. Particularly important
is the fact that in many cases these random fields have an
anisotropic nature in the sense that they have different geometric
characteristics along different directions (see, for example,
Davies and Hall (\cite{DH99}), Bonami and Estrade (\cite{BE03})
and Benson, et al.(\cite{Ben06})).

Moreover, many times the model chosen has to include some
statistical dependence structure that might be present across the
scales. For this purpose, the usual assumption of self-similarity
is formulated. Unfortunately, the classical notion of
self-similarity (see \cite{Lamp62}), defined for a field
$\{X(x)\}_{x\in\mathbb{R}^d}$ on $\R^{d}$ by
$$
\{X(ax)\}_{x \in \R^d} \overset{\mathcal{L}}{=}  \{a^H X(x)\}_{x
\in \R^d}
$$
for some $H \in \R$ (called the Hurst index), is genuinely
isotropic and therefore
has to be changed to fit anisotropic situations.\\

For this reason, there has been an increasing interest in defining
a suitable concept for anisotropic self-similarity. Many authors
have developed techniques to handle anisotropy in the scaling. The
main papers that have to be mentioned in this context are those of Hudson and Mason, Schertzer and Lovejoy (see~\cite{HM82,SL85,SL87}).\\

This motivated the introduction by Bierm\'e, Meerschaert and
Scheffler of operator scaling random fields (OSRF)
in~\cite{BMS07}. These fields satisfy the following scaling
property :
\begin{equation}\label{selfsimilar}
\{X(a^Ex)\}_{x \in \mathbb{R}^d} \overset{\mathcal{L}}{=}  \{a^H
X(x)\}_{x \in \R^d}\;,
\end{equation}
for some $d\times d$ matrix $E$ with positive real parts of the
eigenvalues. \\

A large class of random fields obeys this property. For example
the Fractional Brownian Field (FBM) and the Fractional Brownian
Sheet (FBS) are both Operator Scaling Gaussian Random Fields
(OSGRF) with exponent $E=Id$. Denote $<\cdot,\cdot>$ the Euclidean
scalar product of $\R^d$ defined for any
$x=(x_1,\cdots,x_d)\in\R^d$ and $y=(y_1,\cdots,y_d)\in\R^d$ as
$<x,y>=\sum_{i=1}^d x_i y_i$. Recall that the FBS is the Gaussian
field $\{B_{H_1,\cdots,H_d}(x)\}_{x\in\R^d}$ defined for some
$(H_1,\cdots,H_d)\in (0,1)^d$ as :
\[
B_{H_1,\cdots,H_d}(x)=\int_{\R^d}\frac{(\rme^{\rmi
<x,\xi>}-1)\rmd\widehat{W}(\xi)}{|\xi_1|^{H_1+1/2}\cdots
|\xi_d|^{H_d+1/2}}\;,
\]
where $\rmd\widehat{W}$ is the Fourier transform of white noise on
$\mathbb{R}^d$. This Gaussian field enjoys with the following
scaling property : for all $(a_1,\cdots,a_d)\in (\R_+)^d$
\[
\{B_{H_1,\cdots,H_d}(a_1x_1,\cdots,a_d x_d)\}_{x=(x_1,\cdots,x_d)
\in \mathbb{R}^d}\overset{\mathcal{L}}{=}\{a_1^{H_1}\cdots
a_d^{H_d} B_{H_1,\cdots,H_d}(x_1,\cdots,x_d)\}_{x \in \R^d}\;.
\]
In particular, if we set $a=a_1=\cdots=a_d$, we recover that
\[
\{B_{H_1,\cdots,H_d}(a x_1,\cdots,a x_d)\}_{x=(x_1,\cdots,x_d) \in
\mathbb{R}^d}\overset{\mathcal{L}}{=}\{a^{H_1+\cdots+H_d}
B_{H_1,\cdots,H_d}(x_1,\cdots,x_d)\}_{x \in \R^d}\;,
\]
that is $B_{H_1,\cdots,H_d}$ satisfies
Property~(\ref{EqAutosimHermine}) with $E=Id$ and
$H=H_1+\cdots+H_d$.\\

In~\cite{BMS07} the existence of OSRF with stationary increments
in the stable case for any $d\times d$ matrix $E$ with positive
real parts of the eigenvalues is proved. A special class of OSRF
is defined through its harmonizable representation. For Gaussian
models, which is here the case of interest, it reduces to consider
an integral representation of the form
\[
\int_{\R^d}(\rme^{\rmi <x,\xi>}-1)f^{1/2}(\xi)\rmd
\widehat{W}(\xi)\;,
\]
where $f$ is a positive valued function defined on $\R^d$
satisfying
\[
\int_{\R^d}(1\wedge \|\xi\|^2)f(\xi)\rmd\xi\;<\infty\;,
\]
for any norm $\|\cdot\|$ on $\R^d$. Such a function $f$ is called
{\bf a spectral density}. In order to recover the scaling
property~(\ref{EqAutosimHermine}), the spectral density $f$ is
required to satisfy additional specific homogeneity properties
(see Section~\ref{SecTwo} below). In~\cite{BMS07}, such spectral
densities are defined by an integral formula. This is a non
explicit definition in the sense that actual computations require
numerical approximations. However, these calculations are, in
practice, quite difficult to implement. Nevertheless, a simpler
and explicit
formula is furnished in the particular case of diagonalizable matrices. \\

In this paper, we mainly aim at providing {\bf a complete
description through explicit formulae} for the spectral densities
in the model defined in~\cite{BMS07}. We focus on a specific case
: The {\bf Gaussian} model. The motivation of this restriction is
twofold. On the one hand, it is a reasonable assumption in many
applications; on the other hand, to improve the model, it is
necessary to understand and classify its geometrical properties
which is easier in the Gaussian case.\\

Our main results are stated and proved in Section~\ref{SecFour}.
The first ones, Lemma~\ref{LemCas1}, Lemma~\ref{LemCas2},
Lemma~\ref{LemCas3} and Lemma~\ref{LemCas4}
\begin{enumerate}
\item reduces the construction of an explicit example of OSGRF for
a fixed matrix $E$ and an admissible Hurst exponent $H$ (as
defined in Section~\ref{SecTwo}) to four particular cases related
to specific geometries, \item provides an explicit example in each
of these four specific cases.
\end{enumerate}
Thus, we are able to provide an explicit example of OSGRF
satisfying Equation~(\ref{selfsimilar}) and then extend the
already existing results. Moreover, our second result Theorem
\ref{Ttransfert}, gives a very simple relationship existing
between all possible spectral densities associated to a given
exponent $E$. This result is not formal and can also be turned
into an algorithm which generates different fields -- with
different geometries -- satisfying Equation~(\ref{selfsimilar})
for the same matrix of anisotropy $E$.

These results have important consequences. Firstly, it allows to
define the studied class of OSGRF from four specific cases.
Furthermore we give a complete description of the whole class of
spectral densities of these fields. Finally, since our
construction is explicit, the numerical simulations of OSGRF
become much easier. Thus, our approach provides an explicit
definition of an interesting and large class of fields for
simulations of textures with new geometries. There is actually a
practical motivation to be able to compare natural/real images (of
clouds, bones,...) and models with controlled anisotropy.

In the following pages, we are given
$d\in\mathbb{N}\setminus\{0\}$ and $E$ a $d\times d$ matrix with
positive real parts of the eigenvalues. We define
$$
\lambda_{\min}(E)=\min\limits_{\lambda\in
Sp(E)}(\mathrm{Re}(\lambda))\;.
$$
For any $a>0$ recall that $a^E$ is
defined as follows
\[
a^E=\exp(E\log(a))=\sum\limits_{k\geq
0}\frac{\log^k(a)\;E^k}{k!}\;.
\]
As usual, $E^t$ denotes the transpose of the matrix $E$.\\
We denote $|\cdot|$ the Euclidean norm defined for any
$x=(x_1,\cdots,x_d)\in\R^d$ as
\[
|x|=\left(\sum_{i=1}^d x_i^2\right)^{1/2}\;.
\]

\section{Presentation of the model : Operator Scaling Random Fields (OSRF)}\label{SecTwo}
Let us recall some preliminary facts about Operator Scaling Random
Fields (OSRF) and Operator Scaling Gaussian Random Fields (OSGRF).
We refer to \cite{BMS07} for all the material of this section.
\begin{definition}
A scalar--valued random field $\{X(x)\}_{x\in\mathbb{R}^{d}}$ is
called  operator--scaling if there exists a $d\times d$ matrix $E$
with positive real parts of the eigenvalues and some $H>0$ such
that
\begin{equation}
\{X(a^{E}x)\}_{x\in\mathbb{R}^{d}}\overset{{\mathcal{L}}}{=}\{a^HX(x)\}_{x\in\mathbb{R}^{d}}\;,
\label{EqAutosimHermine}
\end{equation}
where $\overset{(\mathcal{L})}{=}$ denotes equality of all
finite-dimensional marginal distributions. Matrix $E$ and real
number $H$ are respectively called an exponent (of scaling) or an
anisotropy, and an Hurst index of the field.
\end{definition}
\begin{remark} In general, the exponent $E$ and the Hurst index $H$ of an OSRF are not
unique.
\end{remark}
Thus the usual notion of self-similarity is extended replacing
{\bf usual scaling}, (corresponding to the case where $E=Id$)  by
{\bf a linear scaling} involving matrix $E$ (see Figure~1 below).
It allows to define new classes of random fields with new geometry
and structure.
\begin{figure}[H]\label{Fig1}
\begin{minipage}[c]{.32\linewidth}
    \centering
 \includegraphics[angle=0,width=.95\textwidth]{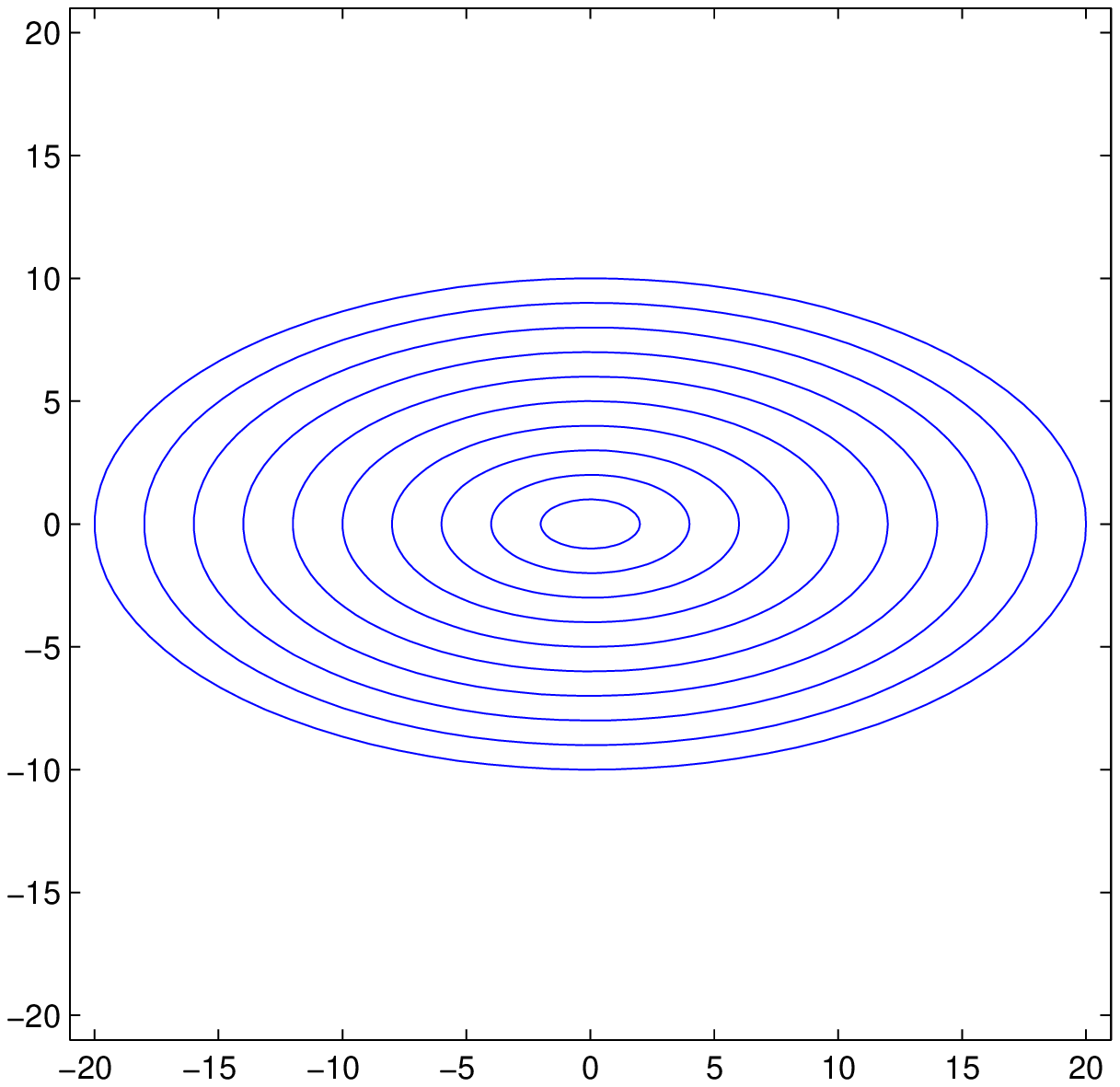}
\begin{center}$E=\begin{pmatrix}1&0\\0&1\end{pmatrix}$, $\lambda\in\{1,\cdots,10\}$\end{center}
\end{minipage}\hfill
  \begin{minipage}[c]{.32\linewidth}
    \hspace{.025\textwidth}
      \includegraphics[angle=0,width=.95\textwidth]{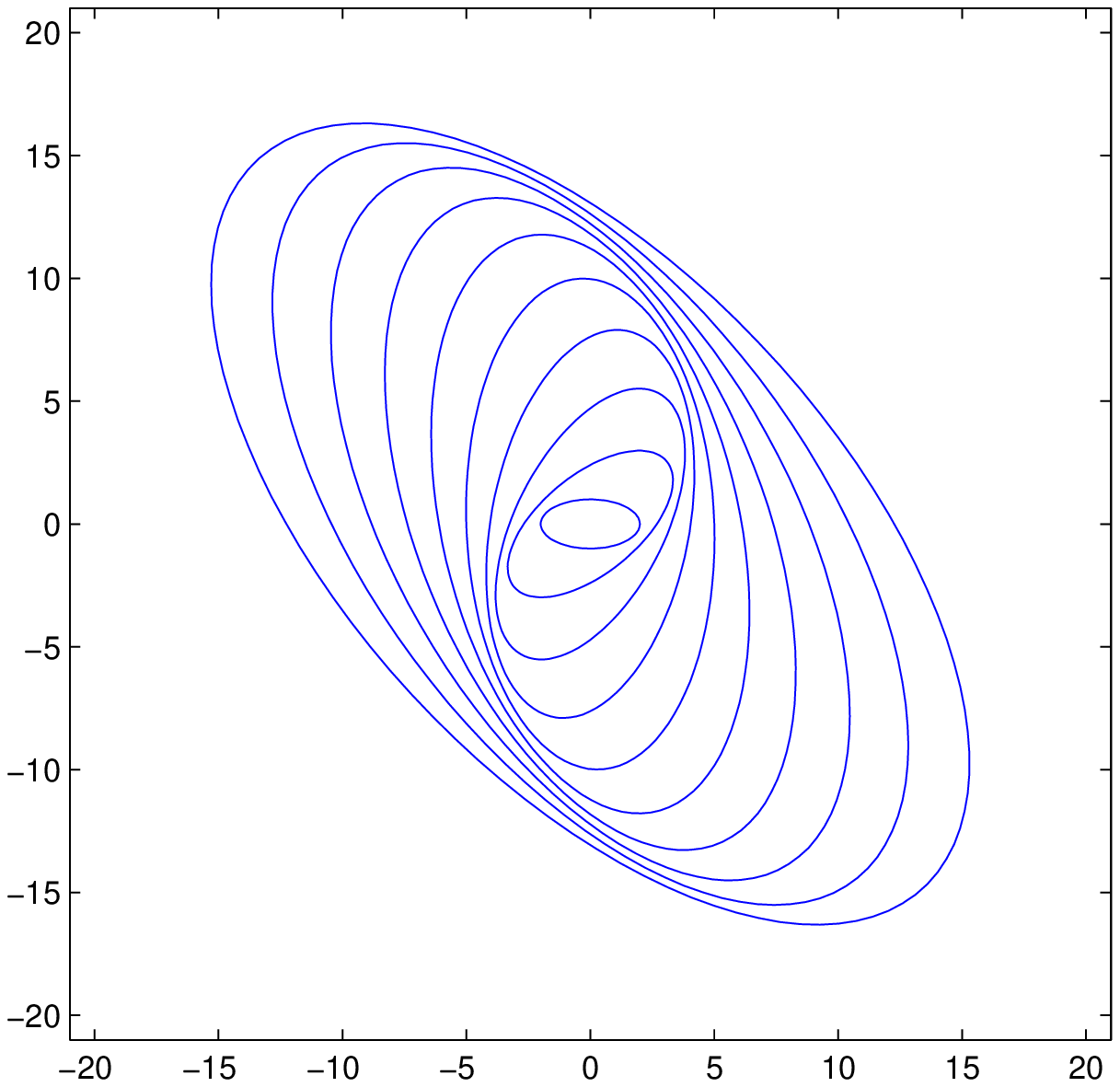}
      \begin{center}$E=\begin{pmatrix}1&-1\\1&1\end{pmatrix}$, $\lambda\in\{1,\cdots,10\}$\end{center}
\end{minipage}
 \begin{minipage}[c]{.32\linewidth}
    \hspace{.025\textwidth}
       \includegraphics[angle=0,width=.95\textwidth]{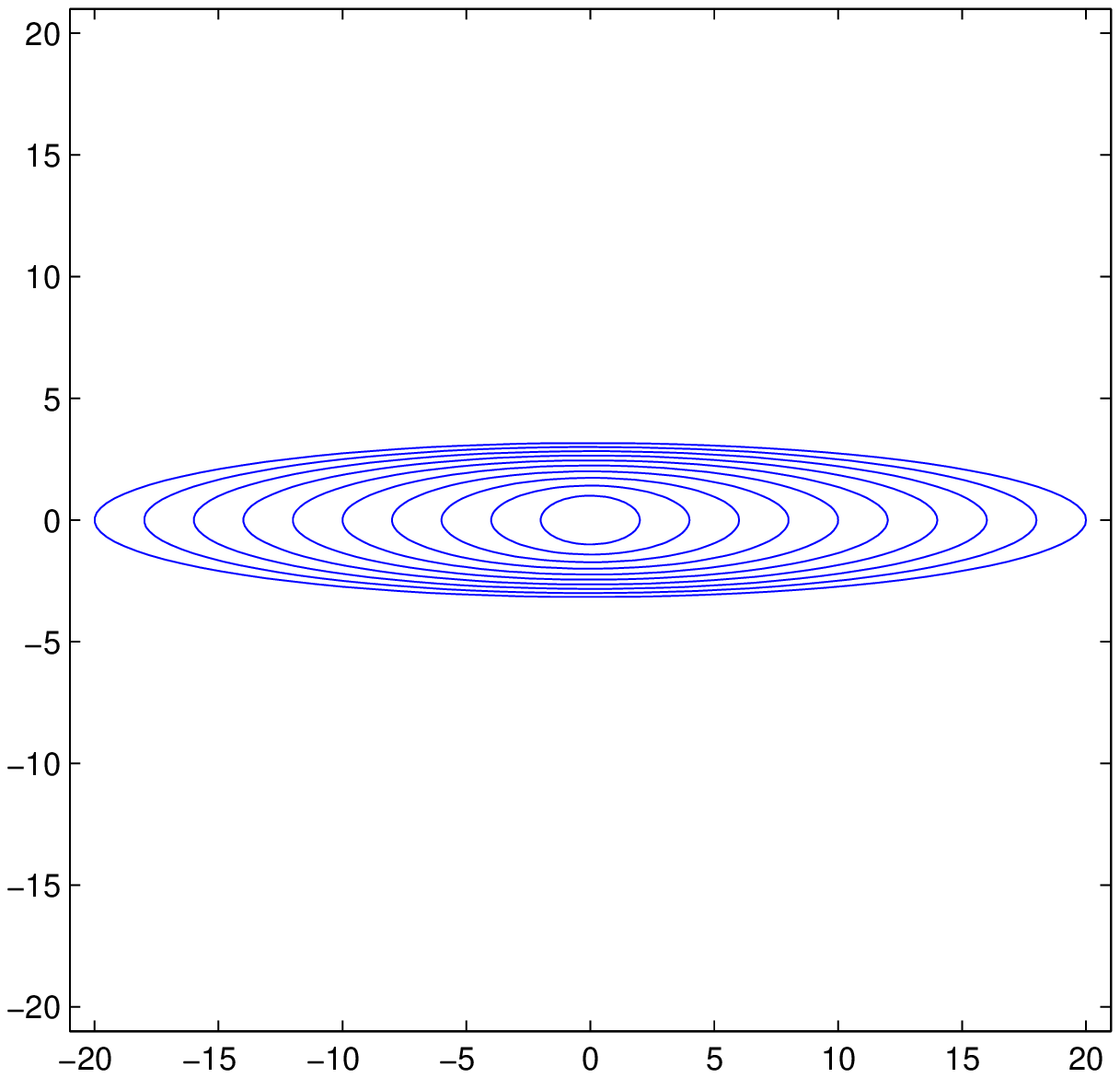}
       \begin{center}$E=\begin{pmatrix}1&0\\0&1/2\end{pmatrix}$, $\lambda\in\{1,\cdots,10\}$\end{center}
\end{minipage}
\begin{center}
\begin{caption} \;Action of a linear scaling $x\mapsto \lambda^{E} x$ on a ellipsis.\end{caption}
\end{center}
\end{figure}

As said in the introduction, when the matrix $E$ is given, the
class of OSRF with exponent $E$ may be very general.
In~\cite{BMS07}, for any given admissible matrix $E$, the
existence of OSRF with
stationary increments is proved using a harmonisable representation.\\

Recall that according to~\cite{Yag57} or \cite{GV67}, given a
stochastically continuous Gaussian field with stationary
increments $\{X(x)\}_{x\in\R^d}$, its covariance can be
represented as
\[
\mathbb{E}(X(x)X(y))=\int_{\R^d}(\rme^{\rmi
<x,\xi>}-1)(\rme^{-\rmi <y, \xi>}-1)\rmd\mu(\xi)+<x, Qy>\;,
\]
where $Q$ is a $d\times d$ non--negative definite matrix and $\mu$
a non--negative symmetric measure $\mu$ such that
\[
\int_{\R^d}(1\wedge \|\xi\|^2)\rmd\mu(\xi)\;.
\]
for any norm $\|\cdot\|$ on $\R^d$. Measure $\mu$ and matrix $Q$
are proved to be unique. Measure $\mu$ is called the spectral
measure of $\{X(x)\}_{x\in\R^d}$. In the case where this measure
is absolutely continuous with respect to Lebesgue measure, the
density function of $\mu$ is called the spectral density of the
Gaussian field $\{X(x)\}_{x\in\R^d}$. In this case, the Gaussian
field $\{X(x)\}_{x\in\R^d}$ can thus be represented as
\begin{equation}\label{EqHarmRepGene}
X(x)\overset{\mathcal{L}}{=}\int_{\R^d}(\rme^{\rmi
<x,\xi>}-1)f^{1/2}(\xi)\rmd\widehat{W}(\xi)\;,
\end{equation}
with
\begin{equation}\label{EqSDGene}
\int_{\R^d}(1\wedge \|\xi\|)f(\xi)\rmd\xi<+\infty\;,
\end{equation}
for any norm $\|\cdot\|$ on $\R^d$. This representation is then
called the harmonisable representation
of the Gaussian field $\{X(x)\}_{x\in\R^d}$.\\

To prove the existence of OSRGF with stationary increments for any
admissible matrix $E$, a quite natural approach is then to use a
harmonisable representation of the form~(\ref{EqHarmRepGene}).
In~\cite{BMS07}, the following result is proved. We state it only
in the Gaussian case :
\begin{theorem}\label{ThDefX}
Let $\rho$ a continuous function defined on $\mathbb{R}^d$ with
non--negative values such that for all $x\in \mathbb{R}^d
\setminus\{0\}$, $\rho(x)\neq 0$. Assume that $\rho$ is
$E^{t}$-homogeneous that is :
\[
\forall a>0,\,\forall
\xi\in\mathbb{R}^{d},\,\rho(a^{E^{t}}\xi)=a\rho(\xi)\;.
\]
Then the Gaussian field $\{X_{\rho}(x)\}_{x\in\mathbb{R}^{d}}$
defined as follows
\begin{equation}\label{EqDefX}
X_{\rho}(x)=\int_{\mathbb{R}^{d}}(\rme^{\rmi<x,\xi>}-1)\rho(\xi)^{-H-\frac{Tr(E)}{2}}\rmd\widehat{W}(\xi),
\end{equation}
exists and is stochastically continuous if and only if $H\in
(0,\lambda_{\min}(E))$. Moreover, this field has the following
properties :
\begin{enumerate}
\item Stationary increments, that is for any $h\in\mathbb{R}^{d}$
\[
\{X_{\rho}(x+h)-X_{\rho}(h)\}_{x\in\mathbb{R}^{d}}\overset{(fd)}{=}\{X_{\rho}(x)\}_{x\in\mathbb{R}^{d}}\;.
\]
\item Operator scaling : The scaling
relation~(\ref{EqAutosimHermine}) is satisfied.
\end{enumerate}
\end{theorem}
\begin{remark}
Through this new class of Gaussian fields, even if it
 is a quite general model, we do not describe the
whole class of OSRGF with stationary increments.
\end{remark}
\begin{remark}
If $H\in (0,\lambda_{\min}(E))$, $f(\xi)=\rho(\xi)^{-2H-Tr(E)}$ is
proved to be a spectral density in the sense that~(\ref{EqSDGene})
holds. Moreover, the spectral density of the Gaussian field
$\{X_{\rho}(x)\}_{x\in\mathbb{R}^{d}}$ is $f$. Observe that $f$ is
continuous and satisfies a specific homogeneity assumption. This
homogeneity assumption is necessary for the operator scaling
property of the Gaussian field $\{X_{\rho}(x)\}_{x\in\R^d}$
whereas the continuity assumption ensures that the field
$\{X_{\rho}(x)\}_{x\in\R^d}$ being defined is stochastically
continuous.
\end{remark}

The main difficulty to overcome is to define {\bf suitable
spectral densities} of this new class of Gaussian fields using
continuous, $E^t$-homogeneous functions with positive values.
In~\cite{PGL94} such functions are called {\bf $(\R^d,E^t)$
pseudo-norms}. They can be defined using an integral formula (see
Theorem~2.11 of~\cite{BMS07}) :
\begin{proposition}
Function $\rho$ defined as
\[
\rho(\xi)=\int_{S_{0}}\int_{0}^{\infty}(1-\cos(<x,r^{E^{t}}\theta>))\frac{\rmd
r}{r^{2}}\rmd\mu(\theta)\;,
\]
is continuous with positive values and $E^t$-homogeneous. Here
$S_{0}$ denotes the unit sphere of $\mathbb{R}^{d}$ for a well
chosen norm defined from $E$ and $\mu$ a finite measure on
$S_{0}$.
\end{proposition}
\begin{remark}Remark that this formula is not the most appropriate
for numerical simulations since we need to approximate an
integral. In what follows, we will give simpler examples of
$(\R^d,E^t)$ pseudo-norms in the sense that these examples lead to
{\bf exact} numerical computations. We then conclude that we give
{\bf explicit} examples of $(\R^d,E^t)$ pseudo-norms (in the
numerical sense).
\end{remark}
\begin{remark}
We also refer to P.G.Lemarie (see \cite{PGL94}) whose definition
of $(\R^d,E^t)$ pseudo-norms is slightly different (see
Remark~\ref{rem:PGL} below).
\end{remark}
Finally, in the special case where matrix $E$ is diagonalizable,
an explicit expression is given (Corollary $2.12$ of~\cite{BMS07})
:
\begin{proposition}
Let $E$ a diagonalizable matrix with positive eigenvalues
$$
0<\lambda_1\leq\cdots\leq \lambda_d,
$$
with associated eigenvectors
$$
\theta_1,\cdots,\theta_d,
$$ and $C_1,\cdots,C_d>0$. Then for any
$\tau<2\lambda_{\min}(E)$
\[
\rho(x)=\left(\sum\limits_{j=1}^{d}C_j|<x,\theta_j>|^{\tau/\lambda_j}\right)^{1/\tau},
\]
is a continuous, $E^t$-homogeneous function with positive values.
\end{proposition}
In this paper, we aim at extending these results and then
describing for any given admissible matrix $E$ and Hurst index
$H$, all the spectral densities of this model of OSGRF with
stationary increments in an {\bf explicit} way.

\section{Definition of explicit spectral densities of the model}\label{SecFour}

As has already been said in Section~\ref{SecTwo}, the main
difficulty is to define appropriate spectral densities of the
model. To this end, we note that the class of the spectral density
used in~\cite{BMS07} is intimately related to the class of the
so-called {\bf pseudo-norms} defined in~\cite{PGL94}.  We then
explicit the link between two $(\R^d,E)$ pseudo-norms when the
matrix $E$ is given. Thereafter using a Jordan reduction, for each
matrix $E$ with positive real parts of the eigenvalues, we give an
explicit example of a suitable spectral density of the studied
model. Combining these two results, we entirely describe {\bf in a
explicit way} the class of spectral densities of the Gaussian
fields considered in~\cite{BMS07}.
\subsection{More about pseudo-norms}\label{SecThree}
Let us first recall some well known facts about pseudo-norms which
can be found with more details in~\cite{PGL94}. This concept is
fundamental when defining anisotropic functional spaces since
using pseudo-norms allows to introduce anisotropic topology on
$\R^d$. Thus, even if the introduction of this concept is not
necessary to the definition of spectral densities, it is of great
importance to relate the notion of anisotropic spectral densities
to the concept of pseudo-norms. This gives us indeed all the tools
of "anisotropic functional analysis" to study, for example, the
sample paths properties of the fields in anisotropic spaces
(see~\cite{CV10}) and to better understand the inherent topology
of these spaces.

\begin{definition}
A function $\rho$ defined on $\R^d$ is a $(\R^d,E)$ pseudo-norm if
it satisfies the three following properties :
\begin{enumerate}
\item $\rho$ is continuous on $\R^d$, \item $\rho$ is
$E$-homogeneous, {\it i.e.} $\rho(a^Ex)= a\rho(x) \quad \forall x
\in \R^d, \, \forall a>0$, \item $\rho$ is positive on $\R^d
\setminus \{0\}$.
\end{enumerate}
\end{definition}
\begin{remark}\label{rem:PGL}
Our definition of $(\R^d,E)$ pseudo-norm is a slightly modified
version of the concept of pseudo-norm on $(\mathbb{R}^{d},A)$
defined by P.G.Lemari\'e in~\cite{PGL94}. In~\cite{PGL94}, $A$
denotes a matrix with eigenvalues having a modulus greater than
one. A pseudo-norm on $(\mathbb{R}^{d},A)$ is a function
satisfying properties $1$ and $3$ of the previous definition and
the following property :
\[
\rho(A x)=|\det(A)|\rho(x),\mbox{ for any }x\mbox{ in
}\mathbb{R}^{d}\;.
\]
Further for any $d\times d$ matrix $A$ with eigenvalues having
modulus greater than one and any compactly supported smooth
function $\phi$, an example of pseudo-norm on $(\mathbb{R}^{d},A)$
is provided by
\[
\rho_{\phi}(x)=\sum\limits_{j\in\mathbb{Z}}|\det(A)|^{j}\phi(A^j
x)\;.
\]
Remark that if $\rho$ is a $(\mathbb{R}^{d},E)$ pseudo-norm then
$\rho(\cdot)^{1/Tr(E)}$ is a pseudo-norm on $(\mathbb{R}^{d},A)$
in the sense of \cite{PGL94} with $A=a^{E}$ for any given $a>0$.
The properties satisfied by $(\mathbb{R}^{d},E)$ pseudo-norms are
very similar to those of pseudo-norms on $(\mathbb{R}^{d},A)$, as
proved in~\cite{PGL94}. Moreover, the example of pseudo-norm on
$(\mathbb{R}^{d},A)$ given in~\cite{PGL94} can be adapted to our
case. Indeed for any compactly supported smooth function $\phi$
\[
\rho_\phi(x)=\displaystyle\int_{0}^{+\infty}\phi(a^{-E}x)\rmd a\;,
\]
is a $(\mathbb{R}^{d},E)$ pseudo-norm. This formula also leads to
numerical approximations and thus is a non explicit one.
\end{remark}
The term of pseudo-norm is justified by the following proposition
which is proved for instance in~\cite{PGL94} or~\cite{BMS07} :
\begin{proposition}
Let $\rho$ a $(\R^d,E)$ pseudo-norm. There exists $C>0$ such that
$$
\rho(x+y) \le C (\rho(x)+\rho(y)), \quad \forall x, \, y \in
\R^d\;.
$$
\end{proposition}
\subsection{Relationship between two given pseudo--norms}
The main result of this section is the description of all the $(\R^d,E)$ pseudo-norms for a given matrix $E$ :
\begin{theorem} \label{Ttransfert}
Let $\rho_1$ be a $(\R^d, E)$ pseudo-norm. Then $\rho_2$ is a
$(\R^d,E)$ pseudo-norm if and only if there exists a continuous
and positive function $g$ defined on $\R^d \setminus \{0\}$ such
that
\begin{equation}\label{EqDefg}
\rho_{2}(\xi)=g(\rho_{1}(\xi)^{-E}\xi)\rho_{1}(\xi)\;.
\end{equation}
\end{theorem}
\begin{proof}
Let $\rho_1$ and $\rho_2$ be two $(\R^d, E)$ pseudo-norms. Then
the function $g= \frac{\rho_2}{\rho_1}$ is continuous, positive on
$\R^d \setminus \{0\}$ and satisfies for all $a>0$,
$$
g(a^E \xi) = g(\xi).
$$
In particular, for a fixed $\xi$ and $a = \rho_1(\xi)^{-1}$, it
follows that
$$
\rho_{2}(\xi)=g(\xi)\rho_{1}(\xi)=g(\rho_{1}(\xi)^{-E}\xi)\rho_{1}(\xi)\;.
$$
The converse is straightforward.
\end{proof}
\\

Consider now the special case $E=Id$. Theorem~\ref{Ttransfert}
implies the following corollary
\begin{corollary}\label{CorIso}
Let $\{X(x)\}_{x\in\R^d}$ be a Gaussian field with stationary
increments admitting a continuous spectral density. Assume that
$X$ is self--similar with Hurst index $H$. Then, there exists a
continuous function $S$ defined on the unit sphere
$\{\xi\in\R^d,\,|\xi|=1\}$ with positive values such that
\begin{equation}\label{EqDefXIso}
\{X(x)\}_{x\in\R^d}\overset{\mathcal{L}}{=}\left\{\int_{\mathbb{R}^{d}}\left(\frac{e^{i<x,\xi>}-1}{|\xi|^{H+d/2}}\right)
S\left(\frac{\xi}{|\xi|}\right)\rmd\widehat{W}(\xi)\right\}_{x\in\R^d}\;.
\end{equation}
\end{corollary}
\begin{proof}
By assumption the Gaussian field with stationary increments
$\{X(x)\}_{x\in\R^d}$ admits a continuous spectral density denoted
$f$. Then
\[
X(x)\overset{\mathcal{L}}{=}\int_{\R^d}(e^{i<x,\xi>}-1)f^{1/2}(\xi)\rmd
\widehat{W}(\xi)\;.
\]
Since $X$ is self--similar with Hurst index $H$,
\[
\{X(ax)\}_{x\in\R^d}\overset{\mathcal{L}}{=}\{a^H
X(x)\}_{x\in\R^d}\;.
\]
By assumption,
\[
X(ax)\overset{\mathcal{L}}{=}\int_{\R^d}(e^{i<ax,\xi>}-1)f^{1/2}(\xi)\rmd
\widehat{W}(\xi)\;.
\]
Set now $\zeta = a\xi$ in the harmonizable representation of $X$
and deduce that
\[
X(ax)\overset{\mathcal{L}}{=}a^{-d/2}\int_{\R^d}(e^{i<x,\zeta>}-1)f^{1/2}(a^{-1}\zeta)\rmd
\widehat{W}(\zeta)\;.
\]
We now identify the two spectral densities of the two Gaussian
fields $\{X(ax)\}_{x\in\R^d}$, $\{a^H X(x)\}_{x\in\R^d}$ which are
equal in law . It implies that
\begin{equation}\label{EqSD}
a^{-d/2}f(a^{-1}\xi)=a^{H}f(\xi)\;,
\end{equation}
that is $\rho(\xi)=f(\xi)^{-1/(H+d/2)}$ is a $(\mathbb{R}^{d},Id)$
pseudo--norm.\\
We now apply Theorem~\ref{Ttransfert} with $E=Id$. Then any
$(\mathbb{R}^{d},Id)$ pseudo--norm $\rho$ can be written
\begin{equation}\label{EqPN}
\rho(\xi)=g(|\xi|^{-1}\xi)|\xi|\;,
\end{equation}
since the Euclidean  norm $|\cdot|$ is a $(\mathbb{R}^{d},Id)$
pseudo--norm. \\
We deduce that any continuous spectral density can be written as
\[
f(\xi)=\left(g(|\xi|^{-1}\xi)|\xi|\right)^{-H-d/2}\;.
\]
Set now $S(\xi)=g(|\xi|^{-1}\xi)^{-H-d/2}$ to deduce the required
result.
\end{proof}
\\

Thus we recover well--known results of Dobrushin
(see~\cite{Dob79}). Indeed, in~\cite{Dob79} a complete description
of self-similar generalized Gaussian fields with stationary
$r$--th increments is given. It implies in particular
Corollary~\label{CorIso}. The class of anisotropic Gaussian field
defined by the representation~(\ref{EqDefXIso}) has been widely
studied (see~\cite{BJR97, BE03}). Recently in~\cite{Ist07}, Istas
has defined an estimator of $S$ using shifted generalized
quadratic variations. Let us emphasize that if an anisotropy $E$
may be known, using Theorem~\ref{Ttransfert} and a fixed
$(\R^d,E^t)$ pseudo--norm $\rho_1$ (see Section~\ref{SecPN}
below), one can probably define in a similar way an estimator of
function
$g$ defined in~(\ref{EqDefg}).\\

In next section, we now define explicit examples of $(\R^d,E^{t})$
pseudo-norms.
\subsection{Explicit construction of $(\R^d,E^{t})$
pseudo-norms}\label{SecPN} The result of this section is based on
the real Jordan decomposition of any $d\times d$ matrix $E$.
\begin{proposition}\label{PropJRed} Any $d\times d$ matrix $E$ can be written, using the real Jordan reduction as
$$
E=
P\begin{pmatrix}E_{1}&&0\\&\ddots&\\0&&E_{m_{1}+m_{2}}\end{pmatrix}P^{-1},
$$
where $(m_1,m_2) \in \N \times \N \setminus \{(0,0)\}$, with
\begin{enumerate}
\item For all $\ell_1 \in \{1, \cdots, m_1\}$,
$$
E_{\ell_1}= \lambda_{\ell_1} Id \text{ or } E_{\ell_1}=
\begin{pmatrix}\lambda_{\ell_{1}}&1&&0\\&\ddots&\ddots&\\&&&1\\0&&&\lambda_{\ell_{1}}\end{pmatrix},
$$
where $\lambda_{\ell_{1}}\in\R$, \item For all $\ell_2 \in
\{1,\cdots, m_2\}$,
$$
E_{m_1+\ell_2}=
\begin{pmatrix}A_{\ell_{2}}&&0\\&\ddots&\\0&&A_{\ell_{2}}\end{pmatrix}
\text{ or }
E_{m_1+\ell_2}= \begin{pmatrix}A_{\ell_{2}}&I_{2}&&0\\&\ddots&\ddots&\\&&&I_{2}\\0&&&A_{\ell_{2}}\end{pmatrix},\\
$$
with $A_{\ell_2}=
\begin{pmatrix}\alpha_{\ell_{2}}&\beta_{\ell_{2}}\\-\beta_{\ell_{2}}&\alpha_{\ell_{2}}\end{pmatrix},
I_{2}=\begin{pmatrix}1&0\\0&1\end{pmatrix}$ where
$(\alpha_{\ell_{2}},\beta_{\ell_{2}})\in\R^2$
\end{enumerate}
\end{proposition}
As a consequence of the real Jordan decomposition, we state the
following proposition :
\begin{proposition}\label{PropClass}The notations are those of
Proposition~\ref{PropJRed}. For any $\ell$, denote $d_{\ell}$ the
size of the matrix $E_{\ell}$. Assume that for each $1 \leq \ell
\le m_1+m_2$, we are given a $(\R^{d_{\ell}},E_{\ell}^t)$
pseudo--norm $\tau_{\ell}$. Define the function $\varphi$ for any
$\xi=(\xi_1, \cdots,
\xi_{m_1+m_2})\in\prod_{\ell=1}^{m_1+m_2}\R^{d_{\ell}}$  as
$$
\varphi(\xi)= \left(\tau_1^2(\xi_1) + \cdots +
\tau_{m_1+m_2}^2(\xi_{m_1+m_2})\right)^{1/2}\;.
$$
Then, the function $\rho$ defined for any $\xi\in\R^d$ as
$$
\rho(\xi) = \varphi(P^t\xi)
$$
is a $(\R^d,E^t)$ pseudo-norm. Further $f=\rho^{-(2H+Tr(E))}$ is a
suitable spectral density of an operator scaling Gaussian random
field with stationary increments.
\end{proposition}
\begin{proof}
Let $F=P^{-1}E P$. Then for any $\zeta=(\zeta_1, \cdots,
\zeta_{m_1+m_2})\in\prod_{\ell=1}^{m_1+m_2}\R^{d_{\ell}}$:
\[
\begin{array}{lll}
\varphi(a^{F^{t}}\zeta)
&=&\left(\tau_{1}^{2}(a^{E_{1}^{t}}\zeta_{1}) + \cdots +
\tau_{m_{1}+m_{2}}^{2}(a^{E_{m_{1}+m_{2}}^{t}}\zeta_{m_{1}+m_{2}})\right)^{\frac{1}{2}}\\
&=&\left(a^{2}\tau_{1}^{2}(\zeta_{1}) + \cdots +
a^{2}\tau_{m_{1}+m_{2}}^{2}(\zeta_{m_{1}+m_{2}})\right)^{\frac{1}{2}}\\
&=&a\varphi(\zeta)
\end{array}
\]
It follows that
\[
\rho(a^{E^{t}}\xi)
=\varphi(P^{t}a^{E^{t}}\xi)=\varphi(a^{F^{t}}P^{t}\xi)=a\varphi(P^{t}\xi)
=a\rho(\xi)
\]
The conclusion is then straightforward.
\end{proof}
\\
Let us illustrate Proposition~\ref{PropClass} through an example :
\begin{example}
Set
\[
E=\begin{pmatrix}2&1\\0&1\end{pmatrix}\;.
\]
Note that $E$ is a diagonalizable matrix since it has two
different eigenvalues. One has $E=PDP^{-1}$ with
\[
D=\begin{pmatrix}2&0\\0&1\end{pmatrix},\,P=\begin{pmatrix}1&-1\\0&1\end{pmatrix}.
\]
A $(\mathbb{R}^{d},D)$ pseudo-norm can be defined as
\[
\rho_{D}(\xi)=|\xi_1|^{1/2}+|\xi_2|\;.
\]
Hence Proposition~\ref{PropClass} allows to give an explicit
expression of a $(\mathbb{R}^{d},E^{t})$ pseudo-norm :
\[
\rho_{E}(\xi)=\rho_{D}(P^{t}\xi)=|\xi_1 |^{1/2}+|\xi_2-\xi_1|\;.
\]
Remark that in this case, Corollary $2.12$ of~\cite{BMS07} exactly
yields the same result since it gives an explicit example of
$(\mathbb{R}^{d},E^t)$ pseudo-norm in the special where matrix $E$
is diagonalizable.
\end{example}
Thus, it is sufficient to define an explicit pseudo-norm for the
four following matrices.
\begin{enumerate}
\item $E_1(\lambda)=
\begin{pmatrix}\lambda&&0\\&\ddots&\\0&&\lambda\end{pmatrix}$, $\lambda\in\R^*_+$.
\item $E_2(\lambda)=
\begin{pmatrix}\lambda&1&&0\\&\ddots&\ddots&\\&&\ddots&1\\0&&&\lambda\end{pmatrix}$,$\lambda\in\R^*_+$.
\item $E_3(\alpha,\beta)=
\begin{pmatrix}A&&0\\&\ddots&\\0&&A\end{pmatrix}$ with $A=
\begin{pmatrix}\alpha&\beta\\-\beta&\alpha\end{pmatrix}$,
$(\alpha,\beta)\in\R^*_+\times \R$. \item $E_4(\alpha,\beta)=
\begin{pmatrix}A&I_{2}&&0\\&\ddots&\ddots&\\&&\ddots&I_{2}\\0&&&A\end{pmatrix}$
with $A=\begin{pmatrix}\alpha&\beta\\-\beta&\alpha\end{pmatrix}$,
$(\alpha,\beta)\in\R^*_+\times \R$.\\
\end{enumerate}
We emphasize that Proposition~\ref{PropClass} above has two
important consequences :
\begin{itemize}
\item The first consequence is that,
Lemmas~\ref{LemCas1},~\ref{LemCas2},~\ref{LemCas3},~\ref{LemCas4},
Proposition~\ref{PropClass} and Theorem~\ref{Ttransfert} give a
complete description of the spectral densities and then of the
class of Gaussian fields introduced in~\cite{BMS07}. \item
Moreover, it implies that all the Gaussian fields belonging to the
class being studied can be generated from {\bf four generic cases}
corresponding to four specific geometries.
\end{itemize}

In the four following lemmas, we define a $(\R^d,E)$ pseudo-norm
in each generic case. Recall that we denote $|\cdot|$ the
Euclidean norm on $\mathbb{R}^d$.\\

Let us first consider the case $E=E_1(\lambda)$ for some
$\lambda\in\R^*_+$:
\begin{lemma}\label{LemCas1}
The function $\rho_1$, defined for $\xi \in \R^d$ by
\begin{equation}\label{EqRho1}
\rho_1(\xi) = \left| \xi \right|^{1/\lambda}\;,
\end{equation}
is a $(\R^d, E_1^t(\lambda))$ pseudo-norm.
\end{lemma}
\begin{proof}
The conclusion is straightforward.\\
\end{proof}
\\
We now consider the case $E=E_2(\lambda)$ for some
$\lambda\in\R^*_+$:
\begin{lemma}\label{LemCas2}
Let us define the functions $\tau_i$ and $\Phi_i$ for any
$i\in\{1,\cdots,d\}$ as follows
\begin{itemize}
\item If $i = 1$, for any $\xi=(\xi_1,\cdots,\xi_d)\in\R^d$,
$\tau_1(\xi)  = \Phi_1(\xi) = \left|\xi_1\right|$. \item If $i \ge
2$, for any $\xi=(\xi_1,\cdots,\xi_d)\in\R^d$
\[
\tau_i(\xi)=
\begin{cases}
\left|\xi_i \right| \quad {\text{ if }}\quad \xi_1= \xi_2=\cdots=\xi_{i-1}=0\;,\\
\Phi_{i-1}(\xi) \left( \Phi_{i-1}(\xi)^{-\lambda^{-1}E^t}
\xi\right)_i \quad \text{otherwise}\;.
\end{cases}
\]
and
\[
\Phi_i(\xi) = \left|\tau_1(\xi)\right|+\cdots +
\left|\tau_i(\xi)\right|\;.
\]
\end{itemize}
Then, the function $\rho_2$ defined for $\xi \in \R^d$ by
\begin{equation}\label{EqRho2}
\rho_2(\xi) = \Phi_d(\xi)^{1/\lambda}\;,
\end{equation}
is a $(\R^d,E_2^t(\lambda))$ pseudo-norm.
\end{lemma}
\begin{proof}
Let us first prove that the function $\rho_2$ is well--defined and
positive on $\R^d \setminus\{0\}$. It is clear that $\Phi_d \ge
0$. Further $\Phi_d(\xi)=0$ if and only if for any
$i\in\{1,\cdots,d\}$, $\tau_i(\xi) =0$. By induction and by
definition of $\tau_i$, it implies that
\[
|\xi_1|=\cdots=|\xi_d|=0\;,
\]
that is $\xi=0$. Therefore, $\rho_2$ is well--defined and positive
on $\R^d \setminus\{0\}$. \\
To show that $\rho_2$ is continuous, the only point to verify is
that for all $1 \le i \le d$, $\tau_i$ is continuous. To this end,
observe that for all $\xi\in\R^d$
\[
\Phi_{i-1} (\xi)  \Phi_{i-1}(\xi) ^{-\lambda^{-1}E_2^t(\lambda)}
\xi = \Phi_{i-1} (\xi)^{Id-\lambda^{-1}E_2^t(\lambda)} \xi= \;.
\]
By definition of the exponential of a matrix, one has
\[
\Phi_{i-1} (\xi)^{Id-\lambda^{-1}E_2^t(\lambda)}
\xi=\xi+\sum_{k=1}^{+\infty}\frac{(-1)^k\log^k(\Phi_{i-1}
(\xi))N^k\xi}{k!}\;,
\]
where $N=\lambda^{-1}E_2^t(\lambda)-Id$. Since
$N=\lambda^{-1}\begin{pmatrix}0&&&0\\1&\ddots&&\\&\ddots&\ddots&\\0&&1&0\end{pmatrix}$,
one has
\[
\Phi_{i-1} (\xi)  \left(  \Phi_{i-1}(\xi) ^{-\lambda^{-1}E^t}
\xi\right)_i = \xi_i
+\sum_{\ell=1}^{i-1}\frac{(-1)^{i-\ell}\log^{i-\ell}(
\Phi_{i-1}(\xi))}{(i-\ell)!\lambda^{i-\ell}}\xi_{\ell} \;.
\]
Further one has by induction on $i\in\{2,\cdots,d\}$, that, for
all $1 \le \ell \le i-1$
$$
\lim_{\xi \to 0} \left(\xi_\ell \log^{i-\ell} (\Phi_{i-1}
(\xi))\right) =0\;.
$$
Then the continuity of $\tau_i$ follows.\\
We now verify that $\rho_2$ satisfies the homogeneity condition.
It can be done by induction on $i$, showing that, for all $1 \le i
\le d$, one has
$$
\tau_i(a^{-E^t_2(\lambda)}\xi) = a^{-\lambda}\tau_i(\xi) \quad
\text{ and}\quad \rho_i(a^{-E^t_2(\lambda)}\xi) = a^{-
\lambda}\Phi_i(\xi)
$$
Indeed, assume that the result holds for $i-1$, then for any $a>0$
and any $\xi$ such that $\xi_1,\cdots, \xi_i$ are not both equal
to $0$ (the other case being trivial),  one has
\begin{eqnarray*}
\tau_i(a^{-E^t_2(\lambda)} \xi)& = &
\vert \Phi_{i-1}(a^{-E^{t}_2(\lambda)}\xi)|(|\rho_{i-1}(a^{-E^{t}_2(\lambda)}\xi)|^{-\frac{E^{t}_2(\lambda)}{\lambda}}a^{-E^{t}_2(\lambda)}\xi)_{i}\\
&=&a^{-\lambda}|\Phi_{i-1}(\xi)|((a^{-\lambda}|\Phi_{i-1}(\xi)|)^{-\frac{E^{t}_2(\lambda)}{\lambda}}
a^{-E^{t}_2(\lambda)}\xi)_{i}\\
&=&a^{-\lambda}\tau_{i}(\xi).
\end{eqnarray*}
We have then proved the homogeneity property of function
$\rho_2$.\\
\end{proof}
\\
We now consider the case $E=E_3(\alpha,\beta)$ for some
$(\alpha,\beta)\in\R^*_+\times \R$:
\begin{lemma}\label{LemCas3}
The function $\rho_3$ defined for $\xi \in \R^d$ by
\begin{equation}\label{EqRho3}
\rho_3(\xi) = \left| \xi \right|^{1/\alpha}\;,
\end{equation}
is a $(\R^d, E_3^t(\alpha,\beta))$ pseudo-norm.
\end{lemma}
\begin{remark} The function $\rho_3$ defined by~(\ref{EqRho3}) is an {\bf isotropic} $(\R^d,E_3^t(\alpha,\beta))$ pseudo--norm,
that is invariant by any isometry $T$ of $\R^d$. Up to a
multiplicative constant, it is the unique one. Indeed, let $\rho$
be another isotropic $(\R^d,E_3^t(\alpha,\beta))$ pseudo--norm.
Observe that for any $a>0$, $a^{E_3^t(\alpha,\beta)}=a^{\alpha}T$
with
\[
T=\begin{pmatrix}R&&0\\&\ddots&\\0&&R\end{pmatrix}\mbox{ where }
R=\begin{pmatrix}\cos(\beta\log(a))&-\sin(\beta\log(a))\\\sin(\beta\log(a))&\cos(\beta\log(a))
\end{pmatrix}\;.
\]
Remark that $T$ is an isometry. By assumptions on $\rho_3$
$$
\rho_3(a^\alpha\xi)=\rho_3(T^{-1} a^{E_3^t} \xi) =
\rho_3(a^{E_3^t} \xi)=a \rho_3(\xi) \quad \forall a>0\;.
$$
Then $\rho_3$ is a $(\R^d, \alpha Id)$ pseudo-norm. Now, consider
the function $g$ defined for any $\xi \in \R^d \setminus \{0\}$ by
$g= \vert \xi \vert^{-1/\alpha}\rho $. Since $\rho_3$ and $\vert
\cdot \vert^{1/\alpha}$ are two isotropic $(\R^d, \alpha Id)$
pseudo-norms, $g$ is isotropic and we have, for all $\xi \in
\R^{d}$ and $a>0$,
\[
g(a^{\alpha} \xi) = g(\xi) \;.
\]
that is setting $b=a^\alpha$, for all $b>0$
\begin{equation}\label{EqHomG}
g(b \xi) = g(\xi) \;.
\end{equation}
Hence $g$ is constant on $\R^d \setminus \{0\}$.
\end{remark}
\begin{remark}
Using Theorem~\ref{Ttransfert}, with $g$ non trivial ({\it i.e.}
non constant on the isotropic unit ball) and $\rho_1= \vert \cdot
\vert^{1/\alpha}$, we are able to define non--isotropic $(\R^d,
E_3^t(\alpha,\beta))$ pseudo-norms.
\end{remark}
\begin{proof}
Observe that for any $a>0$
\[
a^{E_3^t(\alpha,\beta)}=a^{\alpha}\begin{pmatrix}R&&0\\&\ddots&\\0&&R\end{pmatrix}\mbox{
with }
R=\begin{pmatrix}\cos(\beta\log(a))&-\sin(\beta\log(a))\\\sin(\beta\log(a))&\cos(\beta\log(a))
\end{pmatrix}\;.
\]
Since the Euclidean norm $\left|\cdot\right|$ is invariant by any
isometry, and in particular by
\[
T=\begin{pmatrix}R&&0\\&\ddots&\\0&&R\end{pmatrix}\;,
\]
one has for any $\xi\in\R^d$
\[
|a^\alpha T\xi|=|a^\alpha\xi|=a^\alpha|\xi|\;.
\]
The conclusion is then straightforward.\\
\end{proof}
\\
We now consider the case $E=E_4(\alpha,\beta)$ for some
$(\alpha,\beta)\in\R^*_+\times \R$
\begin{lemma}\label{LemCas4}
Let us define the functions $\tau_i$ and $\Phi_i$ for any $1 \le i
\le d$ as ,
\begin{itemize}
\item If $i = 1$, for all $\xi= (\xi_1, \cdots, \xi_d) \in \R^d$,
$\Phi_{1}(\xi)=\tau_{1}(\xi)  =
(|\xi_{1}|^{2}+|\xi_{2}|^{2})^{\frac{1}{2}}$. \item  If $i \geq
2$, for all $\xi= (\xi_1, \cdots, \xi_d) \in \R^d$
\[
\tau_{i}(\xi)  =
\begin{cases}
(|\xi_{2i-1}|^{2}+|\xi_{2i}|^{2})^{\frac{1}{2}}  \quad {\text{ if }}
\xi_{1}= \xi_{2}=\cdots=\xi_{2i-2}=0 &\\
 \Phi_{i-1}(\xi)\left[\left(\Phi_{i-1}(\xi)^{-\alpha^{-1} E^{t}}\xi\right)^{2}_{2i-1}
+\left(\Phi_{i-1}(\xi)^{-\alpha^{-1}E^{t}\xi}\right)^{2}_{2i}\right]^{\frac{1}{2}}
\quad  \text{otherwise}\;, &
\end{cases}
\]
and
\[
\Phi_{i}(\xi) = \left| \tau_{1}(\xi)\right|+\cdots + \left|
\tau_{i}(\xi)\right|\;.
\]
\end{itemize}
Then, the function $\rho_4$ defined for all $\xi \in \R^d$ by
\begin{equation}\label{EqRho4}
\rho_4(\xi) = \Phi_d(\xi)^{1/\alpha}\;,
\end{equation}
is a $(\R^d,E_4(\alpha,\beta))$ pseudo-norm.
\end{lemma}
\begin{proof}
The proof is similar to this of Lemma~\ref{LemCas2}. Indeed, set
for all $1 \le i \le d/2$
$$
r_{i}(\xi)=\left(|\xi_{2i-1}|^{2}+|\xi_{2i}|^{2}\right)^{\frac{1}{2}},
$$
and observe that for all $\xi \in \R^d$
\[
\rho_4(\xi)=\rho_2(r_1(\xi),\cdots,r_{d/2}(\xi))\;,
\]
where $\rho_2$ is defined by~(\ref{EqRho2}) with $\lambda=\alpha$.
\end{proof}
\subsection{Two dimensional examples}
We now focus on the two dimensional case. Up to a change of basis,
$E$ is a matrix of the form :
\begin{enumerate}
\item $E_1(\lambda_1,\lambda_2)=
\begin{pmatrix} \lambda_1 & 0 \\ 0 & \lambda_2 \end{pmatrix}$ with $(\lambda_1,\lambda_2)\in (\mathbb{R}^{*}_{+})^2$.\\
\item $E_2(\lambda)=\begin{pmatrix} \lambda & 0 \\ 1 & \lambda \end{pmatrix}$ with $\lambda\in\mathbb{R}^{*}_{+}$.\\
\item  $E_3(\alpha,\beta) = \begin{pmatrix} \alpha & \beta
\\-\beta & \alpha
\end{pmatrix}$ with $(\alpha,\beta)\in\mathbb{R}^{2}$.
\end{enumerate}
Let us remark that in dimension $2$ there is not four generic
cases but three since the matrix $E$ cannot be equivalent to
$\begin{pmatrix}A&I_{2}&&0\\&\ddots&\ddots&\\&&\ddots&I_{2}\\0&&&A\end{pmatrix}$
with $A=\begin{pmatrix}\alpha&\beta\\-\beta&\alpha\end{pmatrix}$ for some $(\alpha,\beta)\in\R^2$.\\

We now give an explicit example in each of the case above using
the results of Lemma~\ref{LemCas1}, Lemma~\ref{LemCas2} and
Lemma~\ref{LemCas3}:
\begin{enumerate}
\item If $E=E_1(\lambda_1,\lambda_2)$ for some
$(\lambda_1,\lambda_2)\in (\mathbb{R}^{*}_{+})^2$, the function
$\rho_1(\xi_1, \xi_2)= (\vert \xi_1 \vert^{2/\lambda_1}+ \vert
\xi_2 \vert^{2/\lambda_2})^{1/2}$ is a $(\R^2,E_1^t)$
pseudo-norm.\\
\item If $E= E_2(\lambda)$ for some $\lambda\in\R^*_+$, the
function $\rho_2(\xi_1, \xi_2) = (\vert \xi_1 \vert + \vert \xi_2
- \frac{\xi_1}{\lambda} \ln\vert \xi_1 \vert )^{1/\lambda}$
is a $(\R^2,E_2^t)$ pseudo-norm.\\

\item If $E= E_3(\alpha,\beta)$ for some $(\alpha,\beta)\in\R^2$,
the function $\rho(\xi_1, \xi_2 ) = \vert \xi \vert^{1/\alpha} $
is a $(\R^2,E_3^t)$ pseudo-norm. More interesting is the fact that
the function
$$
\rho_{3}=\frac{|\xi_{1}\cos(\beta/\alpha
\ln(r(\xi)))-\xi_{2}\sin(\beta/\alpha
\ln(r(\xi))|}{r(\xi)^{2/\alpha}}\;,
$$
with
\[
r(\xi)=(|\xi_{1}|^{2}+|\xi_{2}|^{2})^{\frac{1}{2}}
\]
is also a $(\R^2,E_{3}^{t})$ pseudo-norm.
\end{enumerate}
Combining these results with Proposition~\ref{PropClass} yields us
to an explicit example of $(\R^2,E^t)$ pseudo-norms for any
$2\times 2$ matrix $E$ whose eigenvalues have positive real parts
(see Figure~2 below). Thereafter Theorem~\ref{Ttransfert} brings
us a complete description of the whole class of $(\R^2,E^t)$
pseudo-norms for any matrix $E$ and thus for spectral densities of
the class of OSGRF defined in~\cite{BMS07}.\\

In Figure~2 just below we represented the pseudo--norms
$\rho_1,\rho_2,\rho_3$ for some values of the matrix $E$. Remark
that the cases $E=\begin{pmatrix}1&0\\0&1\end{pmatrix}$ and
$E=\begin{pmatrix}1&0\\0&1/2\end{pmatrix}$ belong to the same
generic case (the first one).
\begin{figure}[H]\label{Fig3}
  \begin{minipage}[c]{.242\linewidth}
 \includegraphics[angle=0,width=.95\textwidth]{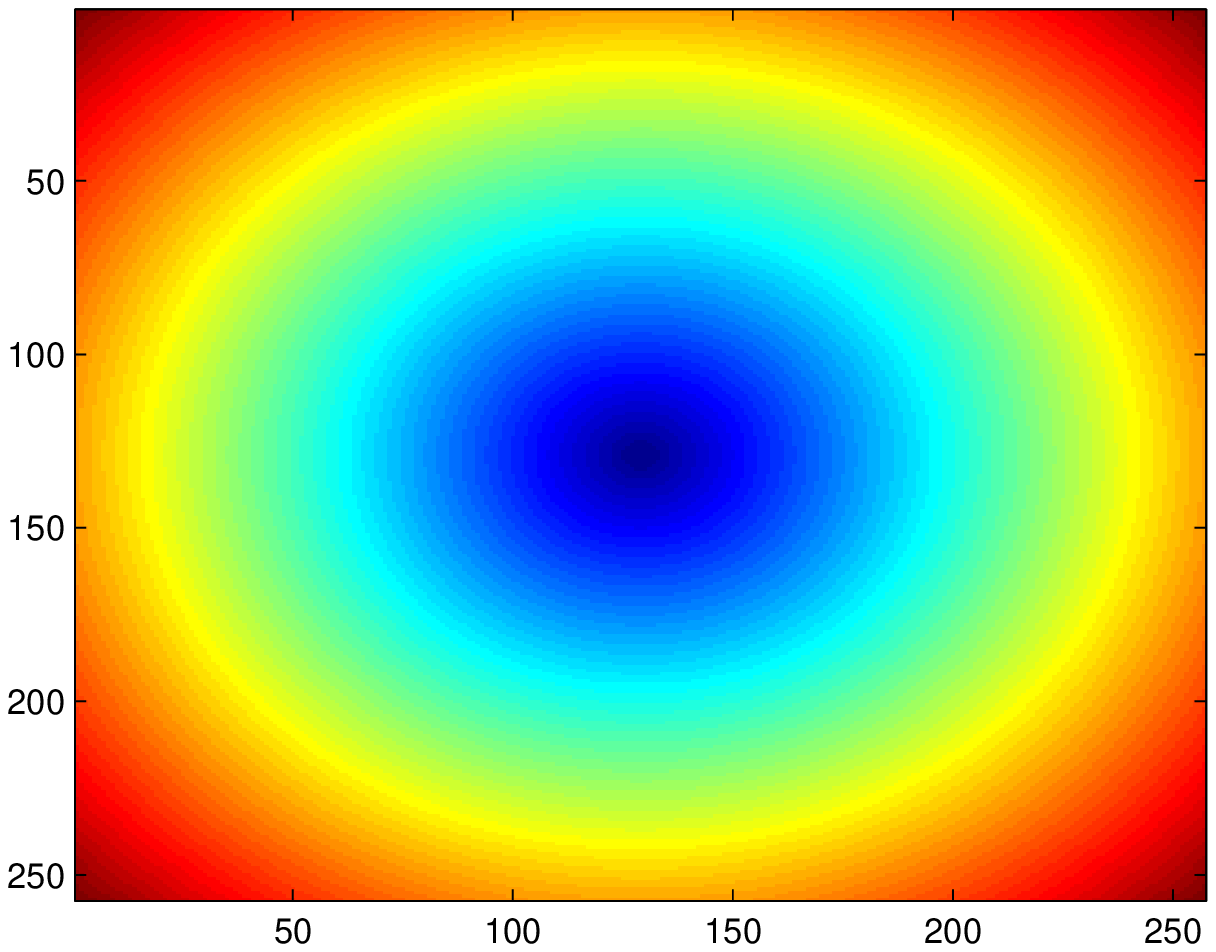}
\begin{center}$E=\begin{pmatrix}1&0\\0&1\end{pmatrix}$\end{center}
\end{minipage}
  \begin{minipage}[c]{.242\linewidth}
    \hspace{.001\textwidth}
      \includegraphics[angle=0,width=.95\textwidth]{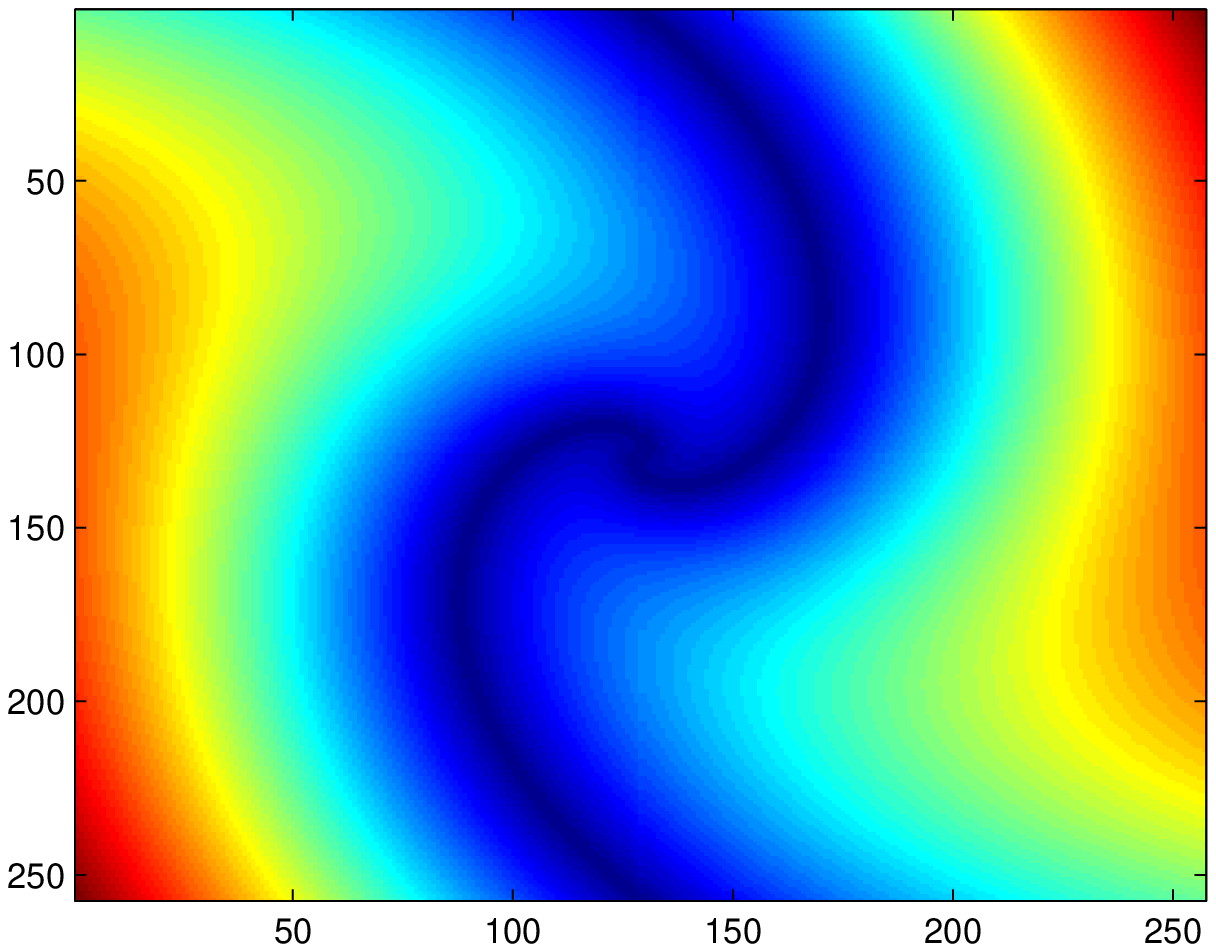}
      \begin{center}$E=\begin{pmatrix}1&-1\\1&1\end{pmatrix}$\end{center}
\end{minipage}
 \begin{minipage}[c]{.242\linewidth}
    \hspace{.001\textwidth}
       \includegraphics[angle=0,width=.95\textwidth]{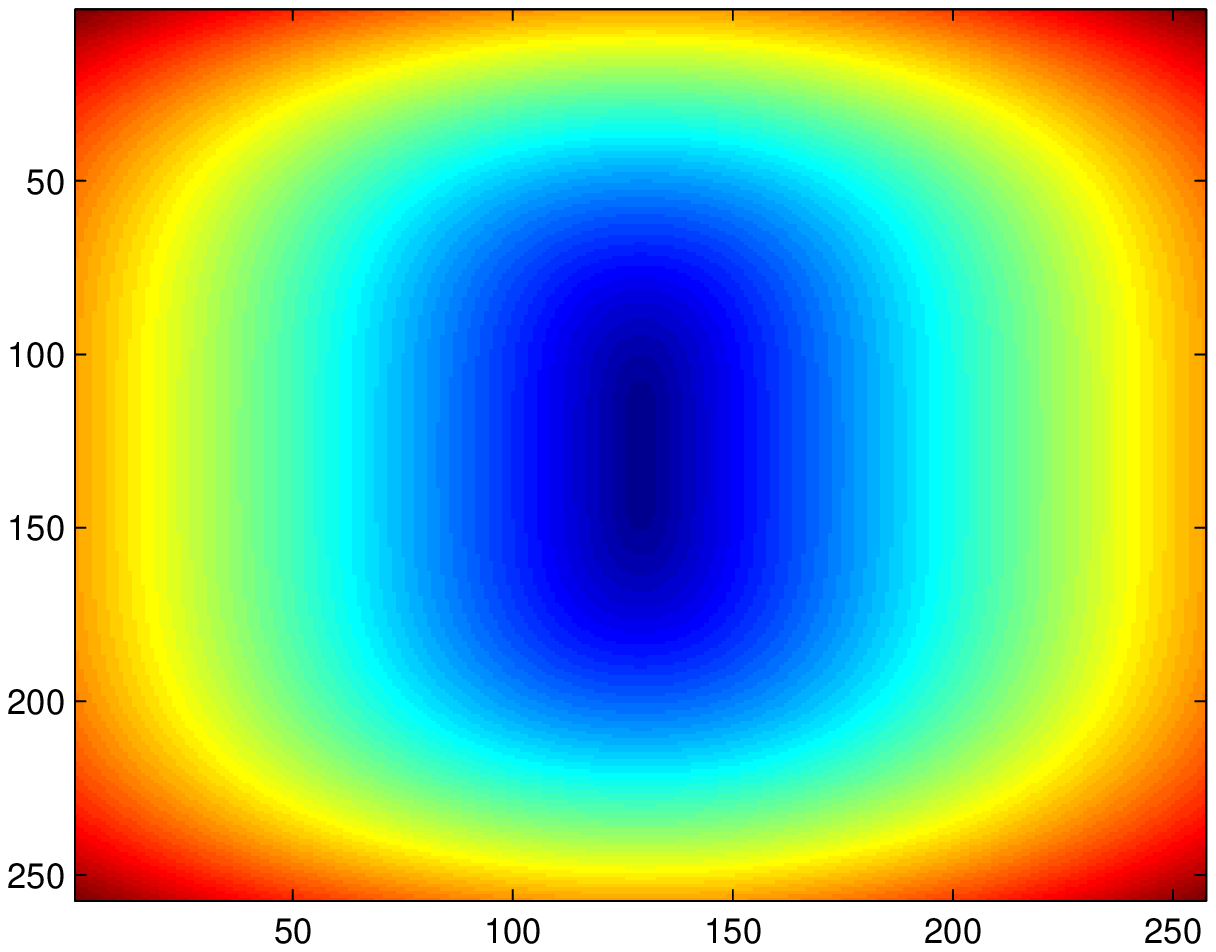}
       \begin{center}$E=\begin{pmatrix}1&0\\0&1/2\end{pmatrix}$\end{center}
\end{minipage}
\begin{minipage}[c]{.242\linewidth}
    \hspace{.001\textwidth}
       \includegraphics[angle=0,width=.95\textwidth]{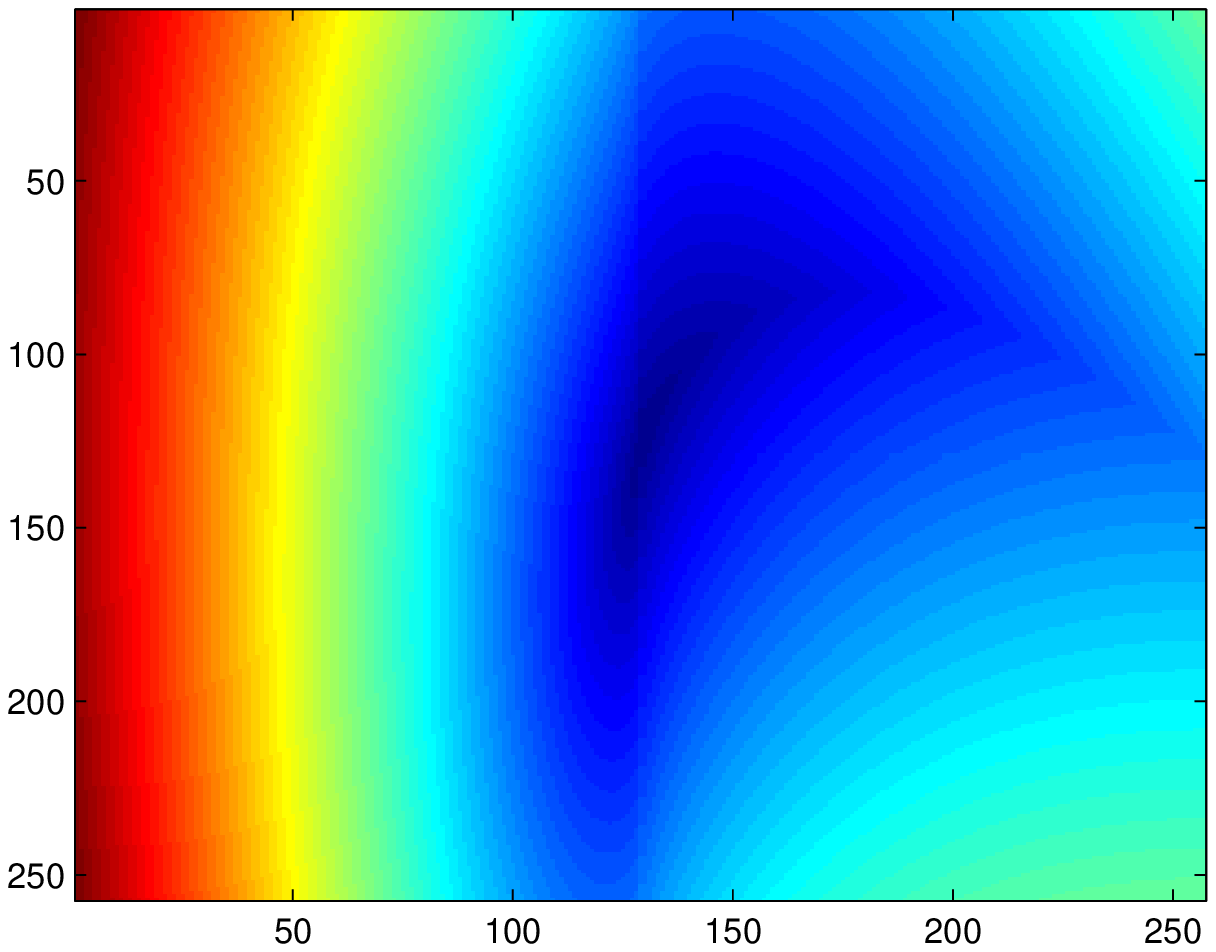}
       \begin{center}$E=\begin{pmatrix}1&1\\0&1\end{pmatrix}$\end{center}
\end{minipage}
\begin{center}
\begin{caption} \;Four pseudo-norms corresponding to the three generic two--dimensional cases.\end{caption}
\end{center}
\end{figure}

\end{document}